\documentclass[11pt]{amsart}
\input{amssym.def}
\input{amssym}
\newtheorem{pro}{Proposition}[section]
\newtheorem{lem}[pro]{Lemma}

\newtheorem{theo}[pro]{Theorem}
\newtheorem{defi}[pro]{Definition}
\newtheorem{cor}[pro]{Corollary}
\newtheorem{remk}[pro]{Remark}

\newcommand{\ep}{\varepsilon}
\newcommand{\al}{\alpha}
\newcommand{\om}{\omega}
\newcommand{\vp}{\varphi}
\newcommand{\la}{\lambda}

\newcommand{\sun}{\odot}

\newcommand{\lra}{\longrightarrow}
\newcommand{\lmt}{\longmapsto}

\newcommand{\nrm}[1]{\mbox{ $ \displaystyle \left\| {#1} \right\| $} }
\newcommand{\nri}[1]{\mbox{ $ \nrm{ {#1} }_{\infty} $} }

\newcommand{\fk}[1]{ \left( {#1} \right) }

\newcommand{\bk}[1]{ \left\{ {#1} \right\} }
\newcommand{\btr}[1]{\mbox{ $ \left| {#1} \right| $ }}

\newcommand{\ce}{{\bf\Bbb C}}

\newcommand{\re}{{\bf\Bbb R}}
\newcommand{\rep}{{\bf\Bbb R^+}}
\newcommand{\za}{{\bf\Bbb N}}

\newcommand{\jz}{{\bf\Bbb J}}

\newcommand{\semig}[2]{\bk{{#1}(t)}_{t\in {#2} }}

\newcommand{\seq}[2]{\mbox{$ \bk{ #1_{#2} }_{{#2} \in \za} $} }

\newcommand{\ilm}[1]{  \lim_{ {#1} \to \infty}  }

\newcommand{\Funk}[5]{ \begin{array}{ccccc}
                       {#1} & : & {#2} & \lra & {#3} \\
                            &   & {#4} & \lmt & \displaystyle{#5} 
                       \end{array}                       }

\begin{document}

\title{Short notes on $L^1(\Omega,X)$ with infinite measure}
\author{Josef Kreulich, \\ Universit\"at Duisburg-Essen}
%\institute{Josef.Kreulich@gmail.com}
\begin{abstract}
This study uses the ideas of \cite{Rieffel} to provide the dual of $L^1(\mu,X)$ in the positive and $\sigma-$ finite cases. This results in elegant necessary and sufficient criteria for weak compactness in $L^1(S,\mu,X)$ in the $\sigma-$finite case, using the ideas of \cite{RuessL1} and \cite{Cooper}. Finally, the result of \cite{NeervenLNM} is extended to compute the sun-dual of $L^1(\re,X)$ with respect to the canonical translation semigroup, dropping the approximation property from $X^*,$, which is applied to obtain almost periodicity for integrals of non-smooth functions. Moreover, for evolution semigroups, it is shown that weak compactness of the orbits implies strong stability.
\end{abstract}

\maketitle

\section{Introduction}
The subject of this study is to extend the results from finite to
infinite measure cases and provide some applications to
$C_0$-semigroups. To obtain the results, we use a representation of
the linear functionals on $L^1(S,X).$ From \cite{CembranoMendoza}, we
know that there are isometries, but the translation of the conditions
does not appear obvious. For this, the duality, weak compactness and an application to $C_0-$semigroup theory are given.

To obtain relatively weak compactness, previously, the conditions in
\cite{DincBrooks} and \cite{Cooper} are provided, which are either not
necessary or aim at a different topology. The proof in this study
follows in the sufficiency part in the idea of \cite{RuessL1} to omit
condition (2) of \cite[Thm. 1]{DincBrooks} and replace it with the
condition on $co\bk{f_n:n\ge k},$, which was used by
\cite{RuessL1}. In the necessity part, an idea given in \cite{Cooper}
is applied. The book \cite{Cooper} is restricted to local compactness,
and the dual Banach space was assumed to have the RNP. Further
relative $\sigma(L^1(\Omega,X^*),L^{\infty}(\Omega,X))-$compactness
was considered. This is surely weaker than the relative
$\sigma(L^1(\Omega,X^*),L^{\infty}(\Omega,X^{**}))-$compactness if the
underlying range space is a dual. Additionally, it is well known that in the general Banach space case, it is not sufficient to consider the functionals coming with $ L^{\infty}(\Omega,X^{**}).$
Moreover, we restrict to one arbitrary chosen coverage $\bigcup_{n=1}^{\infty}A_n=\Omega.$ Under these prerequisites, we will provide necessary and sufficient criteria on relatively weak compactness for general Banach spaces $X,$ and infinite but $\sigma-$finite and positive measure spaces $(\Omega,\Sigma,\mu). $ The results above are applied with the use of sun-dual semigroups. Recall that for a $C_0-$semigroup,
$$
X^{\sun}=\bk{x^*\in X^*:\lim_{h\to 0}T^*(h)x^*=x^*}.
$$
The translation semigroup on the bounded uniformly continuous functions is a well-known $C_0-$semigroup. Using the representation of the linear functionals on $L^1(\re,X),$, it was shown that in general, $BUC(\re,X^*)=L^1(\re,X)^{\sun}.$
This was previously proven under the assumption of $X^*$ having the
approximation property \cite[Ass. 7.3.3,
pp. 131-132]{NeervenLNM}. Additionally, an application of the weak
compactness result gives the equality of weak and strong stability for
so-called evolution semigroups, compare \cite{ChiconeLatushkin}.

\section{Representations of linear functionals and consequences}
Throughout this study, $\mu:\Sigma\to \re$ is assumed to be a positive and $\sigma-$finite measure.

\begin{defi}
Let $X$ be a Banach space and $\Gamma\subset X^*$ and $(\Omega,\Sigma,\mu)$ be a measure space. Then, $u:\Omega\to X$ is called $\Gamma-$measurable if $x^*(u(\cdot)):\Omega \to \ce$ is $\mu-$measurable. In the case of $X=Y^*$ and $\Gamma=X,$, $X-$measurability is called $w^*-$measurability.
\end{defi}
We cite the general result for $p\in[1,\infty)$, but we provide a proof for $p=1$ using \cite[Thm. 5.1]{Rieffel}. This will provide the main ideas and is sufficient for usage in this study.

\begin{theo}\cite[Thm 1.5.4, p. 23]{CembranoMendoza} \label{L-1-X-dual}
Let $X$ be a Banach space and $\Omega$ be a finite measure
space. $(1\le p< \infty),$ and $q$ is the conjugate of $p$. Then, for each $\vp\in L^p(\Omega,\Sigma,\mu,X)^*$,
we have for some $g:\Omega \to X^*$ with
\begin{enumerate}
\item $g$ is $w^*-$measurable.
\item The function $\bk{ \om \to \nrm{g(\om}}\in L^q(\mu) $.
\item $$\vp(f)=\int_{\Omega}<g(\om),f(\om)>_{X^*,X}d\mu(\om)$$.
\item $ \nrm{\vp}=\nrm{\nrm{g(\cdot)}}_q.$
\end{enumerate}
\end{theo}

Using some standard arguments from topology and functional analysis,
we obtain the following.

\begin{cor} \label{sigma-finite-duality}
Let $X$ be a Banach space and $(\Omega,\Sigma,\mu)$ be a $\sigma-$finite measure space. Then, for each $\vp\in L^1(\Omega,\Sigma,\mu,X)^*,$
for some $g:\Omega \to X^*$, we have
\begin{enumerate}
\item $g$ is $w^*-$measurable.
\item The function $\bk{ \om \to \nrm{g(\om}}\in L^{\infty}(\mu) $.
\item $$\vp(f)=\int_{\Omega}<g(\om),f(\om)>_{X^*,X}d\mu(\om)$$.
\item $ \nrm{\vp}=\nrm{\nrm{g(\cdot)}}_{\infty}.$
\end{enumerate}
\end{cor}

Next, we want to consider weak compactness in this measure space, which leads to the following definitions.

\begin{defi}
Let $X$ be a Banach space and $(\Omega, \Sigma,\mu)$ be a finite measure space. Then, $H\subset L^1(\mu,X)$ is called uniformly integrable if
$$
\lim_{\mu(E)\to 0}\int_E\nrm{f}d\mu=0 \mbox{ uniformly for } f\in H.
$$
\end{defi}

Additionally, we generalize equi-integrability, which is considered in the scalar-valued case \cite[IV.15.54, p. 547]{DS} and in the infinite-dimensional case \cite{Cooper}.
\begin{defi} \label{A-equi-integrable}
Let $H\subset L^1(\Omega,X).$
\begin{enumerate}
\item Then, $H$ is $\mathcal{A}-$equi-integrable if there exists $\sigma-$finite coverage $\mathcal{A}=\seq{A}{n},$ with $A_n\subset A_{n+1},$ $ \Omega=\bigcup_{k=1}^{\infty}A_k,$ such that
\begin{enumerate}
\item $H$ is uniformly integrable,
\item there exist $\seq{a}{n}\in l^1(\za)^+,$ such that for all $f\in
  H,$ and
$$\int_{A_{n} \backslash A_{n-1}}\nrm{f}d\mu \le a_n.$$
\end{enumerate}
\item Then, $H$ is $\Omega-$equi-integrable if for every $\sigma-$finite coverage $\mathcal{A}=\seq{A}{n},$ with $A_n\subset A_{n+1},$ $ \Omega=\bigcup_{k=1}^{\infty}A_k,$ we have
\begin{enumerate}
\item $H$ is uniformly integrable,
\item there exist $\seq{a}{n}\in l^1(\za)^+,$ such that for all $f\in
  H,$ and
$$\int_{A_{n} \backslash A_{n-1}}\nrm{f}d\mu \le a_n.$$
\end{enumerate}
\end{enumerate}
\end{defi}

\begin{remk} \label{A-equi-integrable-equiv}
Let $\mathcal{A}=\seq{A}{n},$ $\sigma-$finite coverage with $A_n\subset A_{n+1},$ $ \Omega=\bigcup_{k=1}^{\infty}A_k.$ Then,
$$
\sum_{n=k}^{\infty}\int_{A_{n} \backslash A_{n-1}}\nrm{f}d\mu =\int_{\Omega\backslash A_k}\nrm{f}d\mu.
$$
\end{remk}

Following the proof of \cite[Thm 2.1]{RuessL1} and \cite{Cooper} obtains the elegant result for relatively weak compactness in $L^1(\Omega,\Sigma,\mu,X).$
\begin{theo} \label{sigma-finite-weak-compactness}
Let $\mu$ be a positive Borel measure on a $\sigma-$finite measure space and $H$ be a bounded subset of $L^1(S,\mu,X).$ Then, the following are equivalent:
\begin{enumerate}
\item H is weakly relatively compact.
\item H is $\mathcal{A}-$equi-integrable, and given any sequence $\seq{f}{n}\subset H,$, there exists a sequence $\seq{g}{n}$ with $g_n\in co\bk{f_k:k\ge n}$ such that $\bk{g_n(\omega)}_{n\in \za}$ is norm convergent in $X$ for a.e. $\om\in\Omega$
\item  H is $\mathcal{A}-$equi-integrable, and given any sequence $\seq{f}{n}\subset H,$, there exists a sequence $\seq{g}{n}$ with $g_n\ \in co\bk{f_k:k\ge n}$ such that $\bk{g_n(\omega)}_{n\in \za}$ is weakly convergent in $X$ for a.e. $\om\in\Omega$
\end{enumerate}
\end{theo}

\begin{remk}
In the previous Thm \ref{sigma-finite-weak-compactness}, $\mathcal{A}$-equi-integrability can be replaced by $\Omega-$equi-integrability. As the proof of (1) implies (3), the coverage is arbitrary.
\end{remk}

\begin{remk}
Surely, the finite measure case is a consequence when choosing $A_n=\Omega$ for all $n\in \za.$
\end{remk}

\section{Applications}
To discuss the translation semigroup, recall the following conclusion from the representation of $L^1(\re,X)^*.$
\begin{remk} \label{translation-semigroup-definition}
Let $\jz\in\bk{\rep,\re}$ and $\mu$ be Lebesgue measures such that for all $f\in BUC(\jz),$
the duality of $(L^1(\rep,X),BUC(S,X^*))$ is given by
$$
<g,f>=\int_{\jz}<g(r),f(r)>d\mu(r).
$$
For $t\in \jz$, the translation operator
$$
\Funk{T(t)}{L^{\infty}_{w^*}(\jz,X^*)}{L^{\infty}_{w^*}(\jz,X^*)}{f}{\bk{s\mapsto f(t+s)}, }
$$
is the dual operator to
$$
\Funk{V(t)}{L^1(\jz,X)}{L^1(\jz,X)}{f}{\bk{ s\mapsto \left\{
\begin{array}{rcl} 
f(s-t) &:& s\in  t+\jz \\ 
0 &:& \mbox{ otherwise}
\end{array} \right. }.}
$$
As $\semig{V}{\jz}$ is a $C_0-$semigroup, it has a generator $B,$ and $V^*(t)=T(t).$

\end{remk}

Using the representation of the linear functionals, we extend the
result of \cite{NeervenLNM} for the translation semigroup
$L^1(\re,X)^{\sun}=BUC(\re,X^*),$, which was proven under the
condition that $X^*$ possesses the a.p. %%EDITOR'S NOTE: Abbreviations and acronyms are typically defined the first time the term is used within the main text and then used throughout the remainder of the manuscript. Please consider adhering to this convention. The target journal may have a list of abbreviations that are considered common enough that they do not need to be defined.
In the following, we drop this condition.

\begin{lem} \label{dual-translation-semigroup}
Let $X$ be a Banach space $\semig{V}{\re}$ and the translation
semigroup on $L^1(\re,X).$ %%EDITOR'S NOTE: Please ensure that the
                           %%intended meaning has been maintained in
                           %%this edit.
Then, $L^1(\re,X)^{\sun}=BUC(\re,X^*).$
\end{lem}

\begin{proof} [Proof of Cor. \ref{dual-translation-semigroup}]
By the previous theorem, we have $L^1(\re,X)^*\subset L^{\infty}_{w^*}(\re,X^*).$ Further, $z\in X^*,$ for $g\in L^1(\re,X)^{\sun}$
$$
\nrm{<g(\om+h),z>-<g(\om),z>}_{ess-\infty}=\sup_{h\in B_{L^1(\mu)}}\int_{\Omega}\btr{<g(\om+h),z>-g(\om),z>}h(\om)d\omega
$$
for every $z\in X^*.$
Hence, $<g(\cdot),z>$ fulfils the scalar-valued requirement for $L^1(\mu)^{\sun}=BUC(\re),$ \cite[Exa 1.3.9, p.8]{NeervenLNM}, which yields $<g(\cdot),z> $ uniformly continuous for all $z\in X^*$

Hence, we found $g_z\in L^{\infty}_{w^*}(\re,X)$ defined on every $t\in \re,$ with $<g_z,z>\in BUC(\re),$ and $<x,g_{z_1}(\cdot)>=<x,g_{z_2}(\cdot)>=<x,g(\cdot)>, $ a.e for all $y,z_1,z_2\in X.$

We claim that $g\in C(\re,X^*).$ If not, there is a point of discontinuity $t_0$ and null-sequence $\seq{t}{n}\subset \re,$ such that
$$
\inf_{n\in \za} \nrm{g(t_0+t_n)-g(t_0)}\ge \ep >0.
$$
For some $\seq{z}{n}\subset B_X$

\begin{eqnarray*}
\lefteqn{\nrm{g(t_0+t_n)-g(t_0)}-\frac{1}{n}=\btr{<g_{z_n}(t_0+t_n)-g_{z_n}(t_0),z_n>}} \\
&\le & \sup_{h\in B_{L^1(\mu)}}\int_{\Omega}\btr{<g(\om+t_n),z_n>-g(\om),z_n>}h(\om)d\omega \\
&\le& \sup_{z\in B_X,h\in B_{L^1(\mu)}}\int_{\Omega}\btr{<g(\om+h),z>-g(\om),z>}h(\om)d\omega \\
&\le &\sup_{h\in B_{L^1(\mu,X)}}\int_{\Omega}\btr{<g(\om+h),z>-g(\om),h(\om)>}d\omega \to 0,\\ 
\end{eqnarray*}
whereby the convergence comes with the definition of $g\in L^1(\re,X)^{\sun},$, which yields the contradiction.

\end{proof}

Next we present an application of this representation.
Using the definitions in Remark \ref{translation-semigroup-definition}
and applying Corollary \ref{dual-translation-semigroup} and \cite[Thm. 4.6.11, Lemma 4.6.13]{ArendtBatty}, we obtain
\begin{cor}
Assume $c_0\not\subset X^*,$ let  $f\in L^{\infty}_{w^*}(\re,X^*),$  and let
$\semig{V^*}{\re}\subset L(L^{\infty}_{w^*}(\re,X^*))$ be the dual
translation group. Furthermore, if $V^*(t)f-f\in AP(\re,X^*)$, or $V^*(t)R(\la,B)^*f-R(\la,B)^*f\in AP(\re,X^*)$ for all $t\in \re,$ then $R(\la,B)^*f\in AP(\re,X^*).$
\end{cor}
\begin{proof}
Let $f\in L^{\infty}_{w^*}(\re,X^*)$ and $V^*(t)f-f\in AP(\re,X^*).$ Then,
$$
V^*(t)R(\la,B)^*f-R(\la,B)^*f=V^{\sun}R(\la,B^*)f-R(\la,B^*)f=\in AP(\re,X^*)
$$
As $R(\la,B^*)f\in D(A^*)\subset L^1(\re,X)^{\sun}=BUC(\re,X^*).$
Hence, an application of \cite[Thm. 4.6.1, p. 298]{ArendtBatty} yields $R(\la,B^*)f\in AP(\re,X).$
\end{proof}

In a similar way, the result of Kadets can be extended.
\begin{theo}
Assume that $c_0\not\subset X^*,$ and let $f\in L^{\infty}_{w^*}(\re,X^*),$ with $R(\la,B^*)f\in AP(\re,X^*).$ If for all $\la>0$, the integrals $\int_0^tR(\la,B^*)f(s)ds$ are bounded, then
$$
w^*-\int_0^tf(s)ds \in AP(\re,X^*).
$$
\end{theo}

\begin{proof}
$f\in L^{\infty}_{w^*}(\re,X^*).$ Then, $R(\la,B^*)f\in BUC(\re,X^*)$, and by \cite[Lemma 4.6.13, p. 300]{ArendtBatty},
we have $\int_0^tR(\la,B^*)f(s)ds\in AP(\re,X^*)$ for all $\la>0.$  Then,
\begin{eqnarray*}
\lefteqn{<x,\int_0^tR(\la,B)f(r)dr>}\\
&=&\int_0^t\int_0^{\infty}e^{-\la s}<x,f(s+r)>dtdr =\int_0^{\infty}e^{-\la s}\int_0^t<x,f(s+r)>drds\\
&=&\int_0^{\infty}e^{-\la s}\int_s^{s+t}<x,f(r)>drds.
\end{eqnarray*}
As $\bk{s\mapsto w^*-\int_s^{s+t}f(r)dr=\int_0^tV^*(r)fdr(s)} \in BUC(\re,X^*)$, and $V^*(t)$ $w^*-w^*-$continuous, we have for the $w^*-$integrals
\begin{eqnarray*}
\lefteqn{\int_s^{s+t}f(r)dr-\int_r^{r+t}f(r)dr=\int_s^{r}f(r)dr-\int_t^{r+t}f(r)dr} \\
&=&\int_s^{r}f(r)dr-\int_0^{r}f_t(r)dr=\delta_0\fk{(I-V^*(t))(\int_s^{r}f(\cdot+r)dr},
\end{eqnarray*}
and 
$$
\nrm{\delta_0\fk{(I-V^*(t))(\int_s^{r}f(\cdot+r)dr}}\le \nrm{I-V^*(t)}\nri{\nrm{f}}\btr{s-t}
$$
Consequently, $\bk{w^*-\int_s^{s+t}f(r)dr}_{t\in\re}$ is an equi-Lipschitz family, and partial integration gives
$$
AP(\re,X^*)\ni \int_0^{(\cdot)}\la R(\la,B)f(r)dr= \int_0^{\infty}\la e^{-\la s}\int_s^{s+(\cdot)}f(r)drds \to w^*-\int_0^{(\cdot)}f(s)ds,
$$
uniformly for $\la\to \infty,$ which concludes the proof.
\end{proof}

As an application, so-called evolution semigroups are considered. For
a detailed discussion, see \cite{ChiconeLatushkin}.

\begin{theo}
Let $\semig{S}{\rep}\subset L(X)$ be a bounded $C_0-$semigroup and
$$
\Funk{T(t)}{L^1(\re,X)}{L^1(\re,X)}{f}{\bk{s\mapsto S(t)f(t-s)}}
$$
If $\semig{T}{\rep}$ is Eberlein weakly almost periodic (i.e., $\bk{T(t)f}_{t\ge 0}$ is weakly relatively compact for all $f\in L^1(\re,X).$), then $\semig{T}{\rep}$ is strongly stable (i.e., $\ilm{t}T(t)f=0$ for all $f\in L^1(\re,X).$
\end{theo}

\begin{proof}
Consider the simple function $f(t):= \chi_{[-1,1]}(t)x$; then, by the
assumption that $\bk{T(t)f}_{t\ge 0}$ is weakly relatively compact, and the equi-integrability given by Remark \ref{A-equi-integrable-equiv},
$$
\int_{\btr{s}>n}\nrm{S(t)\chi_{[-1,1]}(t-s)x}ds \to 0 \mbox{ uniformly for } t\ge 0.
$$
The identity
\begin{eqnarray*}
\int_{\btr{s}>n}\nrm{S(t)\chi_{[-1,1]}(t-s)x}ds &=&  \nrm{S(t)x} \int_{\btr{t-[-1,1]} \ge n} 1 ds\\
\end{eqnarray*}
and using the uniform in $t\in \re,$, we have $\nrm{S(t)x}\to 0$ Hence, $\semig{S}{\rep}$ becomes strongly stable and therefore $\semig{T}{\rep}$ on the simple function and by its boundedness on all of $L^1(\re,X).$
\end{proof}

\section{Proofs}

\begin{proof}[Proof of Thm. \ref{L-1-X-dual}]
For the case of $p=1$, we present a proof based on \cite[Thm 5.1]{Rieffel}.
For $\vp\in L^1(\Omega,X)^*$, we consider
$$
\Funk{G}{\Sigma}{X^*}{E}{\bk{x\mapsto \vp(x\chi_E)}.}
$$
Since
\begin{eqnarray}\label{bounded}
\nrm{G(E)(x)}=\nrm{\vp(x\chi_E)}&\le& \nrm{\vp}\nrm{x\chi_E}_1\le \nrm{\vp}\nrm{x}\mu(E).
\end{eqnarray}
It follows that $G$ has its values in $X^*$ and is countably additive. To verify the bounded variation, let $\bk{E_i}_{i=1}^n$ be a partition and $\bk{x_i}_{i=1}^n\subset B_X.$ Then,
\begin{eqnarray*}  
\btr{\sum_{i=1}^nG(E_i)(x_i)}&=&\btr{\vp\fk{\sum_{i=1}^n\chi_{E_i}x_i}}\le \nrm{\vp}\nrm{\sum_{i=1}^n\chi_{E_i}x_i}_1 \nonumber \\
&=&\nrm{\vp}\int_{\Omega}\nrm{\sum_{i=1}^n\chi_{E_i}(\om)x_i} d\mu(\om) \le \nrm{\vp}\sum_{i=1}^n\nrm{x_i}\mu(E_i) \nonumber \\
&\le&\nrm{\vp}\mu(\Omega).
\end{eqnarray*}

Following \cite[Thm. 5.1]{Rieffel} with $V=(X^*,w^*),$, it remains to verify in his notation
$$
A_{\Omega}(G)=\bk{\frac{G(F)}{\mu(F)}: F\subset \Omega, \mu(F)>0} \mbox{ is bounded.} $$
By (\ref{bounded}), we have
$
A_{\Omega}(G)\subset \nrm{\vp} B_{X^*}
$
Hence, we find by \cite[Thm. 5.1]{Rieffel} a $w^*-$integrable function $g,$ such that $G(E)=w^*-\int_Egd\mu,$ for
all $E\in \Sigma,$, and the range of $g$ is a subset of $ A_{\Omega}(G)\subset \nrm{\vp} B_{X^*},$, which yields $\nri{g}\le\nrm{\vp}.$
Plainly, $\nrm{g(\cdot)} \in L^{\infty}(\Omega)$, and if
$f\in L^1(\Omega,X)$ is a simple function, then
$\vp(f)=\int_{\Omega}<g,f>_{X^*,X}d\mu.$ Let $\seq{f}{n}$ be a
sequence of simple functions with the limit $f.$ Then, for a subsequence $<g(\om),f(\om)>=\ilm{n}<g(\om),f_{n_k}(\om)> $, the limit is measurable,
$$
\btr{\int_{\Omega}<g,f_{n_k}>d\mu -\int_{\Omega}<g,f>d\mu}\le \nrm{\vp} \int_{\Omega}\nrm{f_n-f}d\mu,
$$
and
$$
\vp(f)=\ilm{k}\vp(f_n)=\ilm{k}\int_{\Omega}<g,f_{n}>d\mu =\int_{\Omega}<g,f>d\mu,
$$
concludes the proof.
\end{proof}

\begin{proof}[Proof of Corollary \ref{sigma-finite-duality}]
Let $\vp:L^1(\Omega,X)\to \ce$ be a linear functional and $\bigcup A_n=\Omega$ be a $\sigma-$finite cover, with
$A_n\subset A_{n+1}$ Then, we find by the previous result, for the restrictions of
$$
\vp_{|A_n\backslash A_{n-1}}=\tilde{\vp_n}:L^1(A_n\backslash A_{n-1},X)\to \ce,
$$
a sequence $\seq{\tilde{g}}{n}: A_n\backslash A_{n-1} \to
\nrm{\vp}B_{X^*},$ such that
$\tilde{\vp}_n (f)= \int_{A_n\backslash A_{n-1}}<\tilde{g}_n(\om),f(\om)>d\mu(\om),$ for all $f\in L^1(A_n,X).$ Note that with $A_{-1}=\emptyset,$
$$
\vp(f\chi_{A_{n}})=\sum_{k=1}^n\vp(f\chi_{A_k\backslash A_{k-1}})=\sum_{k=1}^n\tilde{\vp}(f\chi_{A_k\backslash A_{k-1}})
$$
With this observation, we define
$$
g_n(\om):=\left\{\begin{array}{lcr} \tilde{g}_k(\om)&:& \om \in A_n\backslash A_{n-1}  \\ 0 &:& otherwise,  \end{array} \right.
$$
and the function
$$
g(\om):=\sum_{k=1}^{\infty}g_k(\om).
$$
Due to the disjoint supports of $g_k$, the sum is well defined and converges pointwise $w^*$ to $g,$ within the range of $g$ contained in $\nrm{\vp}B_{X^*}.$ Therefore, $g\in L^{\infty}_{w^*}(\Omega,X^*),$ and
$$\int_{\Omega}\sum_{k=1}^{n}<g_k,f>d\mu=\vp(f\chi_{A_{n}}).$$

For $f\in L^1(\Omega,X)$, we have
\begin{eqnarray*}
\lefteqn{\btr{\vp(f)-\int_{\Omega}<g,f>d\mu}} \\
&\le& \btr{\vp(f)-\sum_{n=1}^{k}\vp_n(f)}+\btr{\sum_{n=1}^{k}\vp_n(f)-\int_{\om}\sum_{n=1}^{k}<g_n,f>d\mu}\\
&&+\btr{\int_{\Omega}\sum_{n=1}^{k}<g_n-g,f>d\mu},\\
&\le& \btr{\vp(f\chi_{\Omega\backslash A_n})}+\int_{\Omega\backslash A_n}\nrm{f}d\mu \le C \int_{\Omega\backslash A_n}\nrm{f}d\mu.
\end{eqnarray*}
for some $C>0,$ which proves the representation. For $\nrm{\vp}=\nri{\nrm{g(\cdot)}},$, apply the Hölder inequality, which concludes the proof.

\end{proof}

\begin{proof}[Proof of Thm. \ref{sigma-finite-weak-compactness} ]
For sufficiency up to the $\mathcal{A}-$equi-integrability estimation, we may precisely follow the proof of \cite{RuessL1}.
We provide the proof (3) that implies (1). Following Eberlein,
Smul'yan, and Grothendieck, we have to show that given $\seq{f}{m}\subset H$ and $\seq{\vp}{n}\subset B_{L^1(\mu,X)^*},$, we have
$$
\al:=\ilm{n}\ilm{m}\vp_n(f_m)=\ilm{m}\ilm{n}\vp_n(f_m)=:\beta,
$$
provided that the iterated limits exit \cite[Cor. 1 of Thm 7]{Groth}. Given such sequences $\seq{f}{m}$ and $\seq{\vp}{n},$ let $\seq{g}{m}$ be the sequence associated with $\seq{f}{m}$ according to (3), $ E\in \Sigma, \mu(E)=0,$ the exceptional subset of $\Omega.$ Define
$$
g(\om):=\left\{\begin{array}{rcl}
w-\ilm{n}g_n(\om)&:&\om\in \Omega\backslash E \\
0&:&\mbox{ otherwise. }
\end{array}\right.
$$
Clearly, $g$ is essentially separably valued and weakly measurable and
hence strongly measurable. Moreover, by its very definition and using Fatou's Lemma and boundedness of $H,$
$$
\int\nrm{g}d\mu \le \int \liminf \nrm{g_n}d\mu\le \liminf \int \nrm{g_n}d\mu <\infty,
$$
so that $g\in L^1(\mu,X).$ Let us show that the sequence $\seq{g}{n}$ converges weakly in $L^1(\mu,X)$ to $g.$
First, note that we can assume $L^1(\mu)$ to be separable. Then, according to Thm. \ref{L-1-X-dual},
the continuous linear functionals $\tilde{h}$ on $L^1(\mu,X)$ are represented by $w^*-$measurable functions $h:\Omega
\to X^*$ such that $\nrm{h(\cdot)}\in L^{\infty}(\mu),$ the pairing being given by

$<\tilde{h},f>=\int<h,f>d\mu, \ f \in L^1(\mu,X).$ Given any such $h,$ uniform integrability of the sequence
$\bk{<h,g_n>}_{n\in\za}$ and the fact that
$<h(\cdot),g_n(\cdot)>\to <h(\cdot),g(\cdot)>$ a.e. $\Omega,$ in
conjunction with Vitali's convergence theorem imply that for every $A_n$
\begin{equation} \label{Vitali-appl}
\int_{A_n}\btr{<h(s),g_n(s)>\to <h(s),g(s)>}d\mu(s)\to 0.
\end{equation}
Thus, for given

\begin{eqnarray*}
\lefteqn{\int_S\btr{<h(s),g_n(s)>- <h(s),g(s)>}d\mu(s) } \\
&\le& \int_{A_n}\btr{<h(s),g_n(s)> - <h(s),g(s)>}d\mu(s)+\int_{S\backslash A_n}\btr{<h(s),g_n(s)> -<h(s),g(s)>}d\mu(s) \\
&\le& \int_{A_n} \btr{<h(s),g_n(s)>- <h(s),g(s)>}d\mu(s)+C \bk{\int_{S\backslash A_n}\nrm{g_n}d\mu+ \int_{S\backslash A_n} \nrm{g}d\mu}.
\end{eqnarray*}
Again Fatou's lemma yields
$$
\int_{S\backslash A_n} \nrm{g}d\mu \le \liminf \int_{S\backslash A_n}\nrm{g_n}d\mu,
$$
and we find an $n\in\za,$ such that
$$
\bk{\int_{S\backslash A_n}\nrm{g_n}d\mu+ \int_{S\backslash A_n} \nrm{g}d\mu} \le \frac{\ep}{2} \mbox{ uniformly for } n\in \za.
$$
An application of $(\ref{Vitali-appl})$ leads to weak convergence in $L^1(S,\mu,X).$
To complete the implication, note that if $h_m=\sum_1^r\al_if_{k_i}$ is any convex combination of $\seq{f}{m}$ and $\vp \in L^1(S,\mu, X)$ such that $\gamma:=\ilm{m}\vp(f_m)$ exists, then
$$
\btr{\gamma-\vp(h_m)}=\btr{\sum_1^r\al_i(\gamma-\vp(f_{k_i})}\le \max\bk{\btr{\gamma.-\vp(f_{k_i})}:1\le i\le r}.
$$
This shows that $\al=\vp(g)=\beta,$, where $\vp$ is a $w^*-$cluster
point of $\seq{\vp}{n},$, thus completing that proof (3) implies (1).

Implication (1) implies (3); note that uniform integrability comes with \cite[p. 177]{DincBrooks}. It remains to verify (2) of the definition \ref{A-equi-integrable}.
Assuming the contrary to part 2 of the definition of $\mathcal{A}-$equi-integrable, for all positive $\seq{a}{n}\in l^1(\za)$, we find an $f$ such that
$$\int_{A_n\backslash A_{n-1}}\nrm{f}d\mu \ge a_n.$$
Choosing $a^n_n=\ep$ and $a^n_k=0,$ for $k\not= n\in \za,$, we find an $f_n\in A,$ such that
$$\int_{A_n\backslash A_{n-1}}\nrm{f}d\mu \ge a_n=\ep.$$

Then using Hahn-Banach, we find $\psi_k\in L^1(A_k\backslash A_{K-1},X)^*$ with norm 1,
$$
a_k\le \int_{A_n\backslash A_{n-1}}\nrm{\tilde{f}_k}d\mu= <\psi_k,f_k>
$$

By Theorem \ref{L-1-X-dual}, we find $g_k\in L^{\infty}_{w^*}(A_k\backslash A_{K-1},X^*),$ such that $\nri{g_k}\le 1,$ and
$$
a_k\le \int_{A_n\backslash A_{n-1}}\nrm{\tilde{f}_k}d\mu= <\psi_k,f_k>=\int_{A_k\backslash A_{k-1}}<g_k,f_k>d\mu.
$$

Now, we consider the mapping
$$
\Funk{T}{L^1(S,\mu,X)}{l^1(\za)}{f}{\bk{\int_{A_k\backslash A_{k-1}}<g_k,f>d\mu}}_{k\in \za}
$$
and
$$
T(f_n)=\bk{\int_{A_k\backslash A_{k-1}}<g_k,f_n>d\mu}_{k\in \za}
$$
which fails to be uniformly summable for $\ep_1=\ep/2.$ However, due to the weak conditional compactness of $\seq{f}{n},$, the sequence $\bk{Tf_n}_{n\in \za}$ is conditionally compact, a contradiction to \cite[[IV.13.3]{DS}.
\end{proof}


\begin{thebibliography}{99}
\bibitem{ArendtBatty}
Arendt, W., Batty, Ch. J.K., Hieber, M. Neubrander, F., Vector-valued Laplace Transforms and Cauchy Problems. Mono. in Math. Birkh\"auser (2000).


\bibitem{CembranoMendoza}
 Cembrano, P. and Mendoza, J., Banach Spaces of Vector-Valued Functions, LNM 1676 (Springer) (1997).
 
\bibitem{ChiconeLatushkin}
Chicone, C., Latushkin, Y., Evolution Semigroups in Dynamical Systems and Differential Equations, Math. Surv. and Mono. 70, MS (1999).

\bibitem{Cooper} 
Cooper, J.B., Saks Spaces and Applications to Nonlinear Analysis, Notas di Mathematica, 64 ( North Holland Studies)  (1978).

\bibitem{RuessL1}
Diestel,J, Ruess, W.M., Schachermayer, W., Weak compactness in $L^1(\mu,X)$, Prc. Amer. Math. J. Vol 118, No. 2, (1993).

\bibitem{DiestelUhl} Diestel, J. and Uhl, J.J. Jr., Vector Measures, AMS Math. Surv. No. 15. (1977). 

\bibitem{DincBrooks}
Dinculeanu,N., Brooks, J.K., Weak Compactness in Spaces of Bochner Integrable Functions and Applications, Adv Math.  24, 172-188, (1977).

\bibitem{DS}
	Dunford, N. and Schwartz, J. Linear Operators, Part 1, Wiley Intersc. Publ. Pure and Appl. Math. Vol. VII 	(1957)

\bibitem{Groth}
  Grothendieck, A., Crit\`eres de compacit\'e dans le espaces fonctionels
  g\'en\'eraux,
  Amer. J. Math. 74(1952), pp. 168-186.

\bibitem{NeervenLNM}
 	van Neerven, J., The Adjoint of a Semigroup of Linear Operators, Lect. Note Math. 1529, Springer (1992)

\bibitem{Rieffel} 
	Rieffel, M. A., The Radon-Nikodym Theorem for the Bochner Integral, Trans. Amer. Mat. Soc., Vol. 131, no. 2, 1968, pp. 466–487. 




\end{thebibliography}
\end{document}